\documentclass[]{interact}
\usepackage[linktoc=page]{hyperref}
\hypersetup{linkcolor  = black}
\hypersetup{citecolor=black} 
\hypersetup{colorlinks=true}
\hypersetup{urlcolor=cyan}
\usepackage{color, colortbl}
\usepackage{graphicx, amsmath, amsthm, amssymb, mathrsfs, amscd}
\usepackage{epstopdf}
\usepackage[caption=false]{subfig}
\usepackage[numbers,sort&compress]{natbib}
\bibpunct[, ]{[}{]}{,}{n}{,}{,}
\makeatletter
\def\NAT@def@citea{\def\@citea{\NAT@separator}}
\makeatother
\theoremstyle{plain}
\newtheorem{theorem}{Theorem}[section]
\newtheorem{lemma}[theorem]{Lemma}
\newtheorem{corollary}[theorem]{Corollary}

\theoremstyle{definition}

\newtheorem{example}[theorem]{Example}
\theoremstyle{remark}

\begin{document}
\title{Location of the Zeros of Certain Complex-Valued Harmonic Polynomials}
\author{
\bigskip \name{Hunduma~L.~Geleta\textsuperscript{1} and Oluma.A~Alemu\textsuperscript{2}}
 \affil{\textsuperscript{1} Department of Mathematics, Addis Ababa University, Addis Ababa, Ethiopia. \\  Email: hunduma.legesse@aau.edu.et \\ \bigskip \textsuperscript{2} Department of Mathematics, Addis Ababa University, Addis Ababa, Ethiopia. \\  Email:oluma.ararso@aau.edu.et}}
\maketitle
\begin{abstract}
Finding an approximate region containing all the zeros of analytic polynomials is a well-studied problem. But the numb er of the zeros and regions containing all the zeros of complex-valued harmonic polynomials is relatively a fresh research area. It is well known that all the zeros of analytic trinomials are enclosed in some annular sectors that take into account the magnitude of the coefficients. Following Kennedy and Dehmer, we provide the zero inclusion regions of all the zeros of complex-valued harmonic polynomials in general, and in particular, we bound all the zeros of some families of harmonic trinomials in a certain annular region.
\end{abstract}
\section{Introduction}
The location of the zeros of analytic polynomials has been studied by many researchers and we refer the reader to Brilleslyper and Schaubroeck \cite{brilleslyper2014locating}, Dhemer \cite{dehmer2006location}, Frank \cite{frank1946location},  Gilewicz and Leopold \cite{gilewicz1985location}, Howell and Kyle \cite{howell2018locating}, Johnson and Tucker \cite{johnson2009enclosing}, Kennedy \cite{kennedy1940bounds},  and Melman \cite{melman2012geometry}. Recently, researching the number of the zeros of general analytic trinomials  and  the regions in which  the zeros are located has become of interest due to their application in other fields. For example, the roots of the trinomial equations can be interpreted as the equilibrium points of unit masses that are located at the vertices  of two regular concentric polygons  centered at   the origin in the complex plane (Szabo \cite{szabo2010roots}). A useful result in determining the location of the zeros of analytic polynomials is the argument principle. Recall that if $f$ is analytic inside and on positively oriented rectifiable Jordan curve C and  $f(z) \neq 0$  on C, then  the winding number of the image curve $f(C)$ about the origin $,\frac{1}{2\pi}\Delta _C~ arg f(z),$  equals the total number of zeros of $f$ in D, counted according to multiplicity where D is a plane domain bounded by C. In 1940 Kennedy  \cite{kennedy1940bounds} showed that  the roots of analytic trinoimial equations of the form $z^n + az^k + b = 0 ,$ where $ab \neq 0,$  have  certain bounds to their respective absolute values. In 2006, Dehmer \cite{dehmer2006location} proved that all the zeros of complex-valued analytic polynomials lie in certain closed disk. In 2012,  Melman \cite{melman2012geometry} investigated the regions in which the zeros of analytic trinomials of the form $p(z) = z^n - \alpha z^k - 1,$ with integers $n \geq 3$ and $1 \leq k \leq n - 1$ with $\mathrm{gcd}(k, n) = 1,$ and $\alpha \in \mathbb{C},$ lie. He determined the zero inclusion regions for the following cases: a$)$ for any value of $|\alpha|;$ $b)$  $| \alpha| > \sigma (n, k);$ and c$)$  $| \alpha| < \sigma (n, k),$  where  $\sigma(n, k)=\frac{n}{n-k} (\frac{n-k}{k})^{\frac{k}{n}}.$ In each cases, Melman provided useful information  on \bigskip the location  of the zeros of $p.$ \\
One area of investigation that has recently become of interest is the number and location of zeros of complex-valued harmonic polynomials. In 1984 Clunie and Sheil-Small \cite{clunie1984harmonic} introduced the family of complex-valued harmonic functions $f = u + iv$ defined in the unit disk $\mathbb{D} = \{z : |z| < 1\},$ where $u$ and $v$ are real harmonic in $\mathbb{D}.$ A continuous function $f= u+iv$ defined in a domain $G \subset \mathbb{C}$ is harmonic in $G$ if $u$ and $v$ are real harmonic in $G.$  In any simply connected subdomain of $G,$ we can decompose $f$ as $ f = h + \overline{g},$  where $g$ and $h$ are analytic and $\overline{g}$ denotes the function $z \mapsto \overline{g(z)}($see Duren \cite{duren2001univalent}). This family of complex-valued harmonic functions is a generalization of analytic mappings studied in geometric function theory, and much research has been  done investigating the properties of these harmonic functions. For an overview of the topic, see Duren \cite{duren2004harmonic} and Dorff and Rolf \cite{dorff2012anamorphosis}. It was shown by Bshouty \textit{et al.} \cite{bshouty1995exact} that there exists  a complex-valued harmonic polynomial $ f = h + \overline{g},$ such that $h$  is an analytic polynomial of degree $n,$ $g$  is an analytic polynomial of degree $m < n$ and $f$ has exactly $n^2$ zeros counting with  multiplicities in the field of complex numbers,  $\mathbb{C}.$ We are motivated by the work of Brilleslyper et al. \cite{brilleslyper2020zeros},  Kennedy \cite{kennedy1940bounds}, and Dehmer \cite{dehmer2006location} on the number and location of zeros of  trinomials. Recently, Brilleslyper et al. \cite{brilleslyper2020zeros} studied the number of zeros of harmonic trinomials of the form $p_c(z) = z^n + c\overline{z}^k - 1$  where $1 \leq k \leq n - 1, n \geq 3, c \in \mathbb{R^+},$ and $\mathrm{gcd}(n, k) = 1.$ They showed that the number of zeros of  $p_c(z)$ changes as $c$ varies and proved that the distinct number of zeros of $p_c(z)$ ranges from $n$ to $n+2k$. Among other things, they used the argument principle for harmonic function that can be formulated as a direct generalization of the classical result for analytic functions (Duren et al. \cite{duren1996argument}). An interesting open problem raised in \cite{brilleslyper2020zeros} was finding where the zeros of this trinomials are  located and deciding whether the \bigskip value of $c$ affects the zero inclusion regions of  $p_c(z).$\\
In this paper, we first look at general harmonic polynomials and locate the zeros by using some results obtained by other mathematicians. Then using these results, we bound all the zeros of $p_c(z)$ by limiting our consideration to different cases and we conclude that the location of the  zeros of  $p_c(z)$ depends on the value of $c.$  The main theorems in this paper   are  Theorem $\ref{a'}$ and Theorem $\ref{a}.$  The structure of our paper is  as follows. In section $\ref{p}$, we present some important preliminary results that will be used to prove our two main theorems. In section $\ref{q''}$, we state and prove the main results of this paper. Theorem $\ref{a'}$ determines an upper bound for the moduli of  all the zeros of complex-valued harmonic polynomials and  provides the zero inclusion regions for harmonic trinomial $p_c(z) = z^n + c\overline{z}^k - 1$ where where $1 \leq k \leq n - 1, n \geq 3, c \in \mathbb{R^+},$ and $\mathrm{gcd}(n, k) = 1.$ By considering different cases,  Theorem $\ref{a}$  describes other  zero inclusion regions, which is  sharper, for the same trinomial $p_c(z) = z^n + c\overline{z}^k - 1$ where $0<c<1, c\geq1$ and $c \in \mathbb{C}.$ 

   \section{Preliminaries}$\label{p}$
   In this section we review some important concepts and results that we will use later on to prove our theorems. We begin by stating the well known results such as Cauchy's bound, Descarte's Rule of Sign and some useful theorems and lemmas.
   \begin{theorem}[\textbf{Cauchy's bound}]$\label{p'}$
   Suppose  $p(x) = a_nx^n + a_{n-1}x^{n-1} + \cdots + a_1x + a_0$ be a polynomial of degree $n$ with $n\geq 1.$  Let $M:= \mathrm{max}\left\{ \left|\frac{a_j}{a_n}\right| \right\}_{j=0}^{n-1}.$ Then all zeros of $p$ lie in the interval $[-(M + 1), M + 1].$
   \end{theorem}
   Cauchy's bound tells us where we can find the real zeros of a polynomial. For instance, the real zeros of the polynomial $p(x) = 5x ^3 - 8 x ^2 +  x - 10$ lie in the interval $[-3,3].$
    \begin{theorem}[\textbf{Descartes' Rule of Signs} \cite{wang2004simple}] $\label{p''}$
    Let $p(x) = a_0x^{b_0}+ a_1x^{b_1}+ \cdots + a_nx^{b_n}$ denote a polynomial with nonzero real coefﬁcients $a_i,$ where the $b_i$ are integers satisfying $0 \leq b_0 < b_1< b_2< \cdots < b_n.$ Then the number of positive real zeros of $p(x)$ (counted with multiplicities) is either equal to the number of variations in sign in the sequence $a_0, \cdots , a_n$ of the coefficients or less than that by an even whole number. The number of negative zeros of p(x) (counted with multiplicities) is either equal to the number of variations in sign in the sequence of the coefficients of $p(-x)$ or less than that by an even whole number.
   \end{theorem}
   Note that as the Fundamental Theorem of Algebra gives us an upper bound on the total number of roots of a polynomial, Descartes' Rule of Signs gives us an upper bound on the total number of positive ones. According to Descartes's Rule of Signs, polynomial $p(x)$ has no more positive roots than it has sign changes and  has no more negative roots than  $p(-x)$  has sign changes. A polynomial may not achieve the maximum allowable number of roots given by the Fundamental Theorem, and likewise it may not achieve the maximum allowable number of positive roots given by the Rule of Signs. 
   \begin{example} 
Consider the polynomial $p(x) = x ^3 - 8 x ^2 + 17 x - 10.$ Proceeding from left to right, we see that the terms of the polynomial carry the signs  $+ - + -$  for a total of three sign changes. Descartes' Rule of Signs tells us that this polynomial may have up to three positive roots and has no negative root. In fact, it has exactly three positive roots $1, 2,$ and  $5.$
   \end{example}
   \begin{theorem}[\textbf{Marden,M.} \cite{marden1949geometry}]$\label{p'''}$
   All the zeros of the polynomial $f(z) = a_o + a_1z + \cdots +a_nz^n,~~ a_n\neq 0,$ lie in the closed disc $|z| \leq r,$ where r is a positive real root of the equation $|a_o| + |a_1| z + \cdots + |a_{n-1}| z^{n-1} -|a_n|z^n =0.$ 
\end{theorem}
Note that Theorem $\ref{p'''}$ is a classical solution in finding an upper bound for the moduli \bigskip of all the zeros of the polynomial, which is analytic in general.\\  The following theorem is also another classical result for the location of the zeros of analytic complex polynomials that depends on algebraic equation's positive root. Descartes' Rule of Signs plays a great role in the proof as shown by Dehmer. 
   \begin{theorem}[\textbf{Dehmer, M.} \cite{dehmer2006location}]$\label{p''''}$
   Let $f(z) = a_nz^n+a_{n-1}z^{n-1}+ \cdots +a_1z+a_0, ~~~ a_n \neq 0$ be a complex polynomial. All zeros of $f(z)$ lie in the closed disc $K(0,\mathrm{max}(1,\delta)),$ where $M:= \mathrm{max}\{ |\frac{a_j}{a_n}| \} ~~\forall j=0,1,2,\cdots ,n-1$ and $\delta \neq 1$ denotes the positive real root of the equation $z^{n+1} - (1+M)z^n +M =0.$ 
   \end{theorem}
\begin{theorem}[\textbf{Rouche's Theorem}]
Suppose $f$ and $g$ are holomorphic functions in the interior of a simple closed Jordan curve $C$ and continuous on $C$ with $| g(z)| < | f(z)|$ for all $z \in C.$ Then $f$ and $f + g$ have the same number of zeros in the interior of the curve, counting with multiplicities.
\end{theorem}
  Not only counting the number of zeros of analytic functions, Rouche's Theorem also plays a great role in bounding all the zeros of analytic functions, as we prove in the following lemma. 
\begin{lemma}
The zeros of complex polynomial $g(z)=z^n+b_{n-1}z^{n-1}+ \cdots +b_1z+b_0$ all lie in the open disk centered at the origin with radius  $R= \sqrt{1+|b_{n-1}|^2+ \cdots +|b_1|^2+|b_0|^2}.$
\end{lemma}
\begin{proof}
 Note that $R=1$ if and only if $g(z)=z^n.$ In this case, the assertion is obviously true. Therefore, we may assume that $R>1.$ If $|z|=R,$ then using the Cauchy-Schwarz inequality \\  $|z^n-p(z)|=|b_{n-1}z^{n-1}+ \cdots +b_1z+b_0| \\ ~~~~~~~~~~~~~~~\leq \left(\sqrt{|b_{n-1}|^2+ \cdots +|b_1|^2+|b_0|^2}\right)\left(\sqrt{R^{2(n-1)}+R^{2(n-2)}+ \cdots +1}\right) \\~~~~~~~~~~~~~~~ = \left(\sqrt{R^2-1}\right)\left(\sqrt{\frac{R^{2n}-1}{R^2-1}}\right)\\ ~~~~~~~~~~~~~~~= \sqrt{R^{2n}-1}<R^n=|z^n|.$ \\ Therefore, by Rouche's Theorem the zeros of $g$ all lie in the open disk centered at the origin with radius  $R.$
\end{proof}
\section{Main Results}$\label{q''}$
Under this section, we bound all the zeros of complex-valued harmonic polynomials and specifically we bound the zeros of some families of trinomials to come up with a certain conclusion.
\subsection{Bounds for the zeros of general harmonic polynomials}$\label{p'}$
In this section we find an upper bound  on the moduli of all the zeros of complex-valued harmonic polynomials.
The following theorem derives a closed disk in which all zeros are included and hence it bounds the moduli of the zeros of the complex-valued harmonic  polynomials. 
 \begin{theorem}$\label{a'}$
 Let $$p(z) = h(z) + \overline{g(z)} = a_0+ \sum_{k=1}^n a_kz^k + \overline{\left(b_0+\sum_{k=1}^m b_kz^k\right)}$$ be a complex-valued harmonic polynomial with $\mathrm{deg}h(z) = n \geq m=\mathrm{deg}g(z)$ and let  $r \neq 1$ be positive real root of equation $~~x^{n+1} -(1+M)x^n+M=0~~$ where $M:= \mathrm{max}\left\{ \frac{|a_j|+|b_j|}{|a_n|}\right\}_{j=0}^{n-1}.$  Then all the zeros of $p(z)$ lie in the closed disk $D(0,R)$ where $R = \mathrm{max}(1, r).$ 
 \end{theorem}
 \begin{proof}
 Observe that, since $\mathrm{deg}g(z)=m,$ we have $b_{m+1} = b_{m+2} = \cdots = b_{n-1} = 0.$ Hence $M = \mathrm{max}\left\{ \frac{|a_j|}{|a_n|}\right\} ~~ \forall j > m .$ Now by triangle inequality we have  $$\left| p(z)\right| = \left| \sum_{k=0}^n a_kz^k + \overline{\left(\sum_{k=0}^m b_kz^k\right)} \right| \geq \left| \left| a_nz^n\right| - \left[\left|\sum_{k=0}^{n-1}a_kz^k + \overline{\left(\sum_{k=0}^m b_kz^k\right)}\right|\right]\right| $$ Since $$\left|\sum_{k=0}^{n-1}a_kz^k + \overline{\left(\sum_{k=0}^m b_kz^k\right)}\right| \leq \sum_{k=0}^{n-1}\left|a_k||z^k \right| + \overline{\left(\sum_{k=0}^m\left|b_k||z^k \right|\right)}$$ and the modulus of any complex number equals with modulus of its conjugate, after some algebraic manipulations we arrive at $$ \left| p(z)\right| \geq |a_n| \left[|z|^n-\left(\frac{|a_{n-1}|}{|a_n|}|z|^{n-1}+ \cdots +\frac{|a_m|+|b_m|}{|a_n|}|z|^m+ \cdots +\frac{|a_0|+|b_0|}{|a_n|}\right)\right].$$ Then we get $$\left| p(z)\right| \geq |a_n|\left[|z|^n-M\left(|z|^{n-1}+ \cdots +1 \right)\right]=|a_n|\left[ \frac{|z|^{n+1}-(1+M)|z|^n+M}{|z|-1} \right].$$ Descartes' Rule of Signs gives us that the function $x^{n+1} -(1+M)x^n+M$ has exactly two distinct positive zeros, since $x=1$ is a root with multiplicity one \cite{wang2004simple}. We then conclude that $|p(z)| \neq 0 ~~ \forall z \in \mathbb{C}-\overline{D(0,R)}.$ Hence all the zeros of $p(z)$ lie in the closed disk $D(0,R)$ where $M:= \mathrm{max}\left\{ \frac{|a_j|+|b_j|}{|a_n|}\right\}_{j=0}^{n-1},$  $r \neq 1$ is a positive real root of equation $~~x^{n+1} -(1+M)x^n+M=0~~$ and $R = \mathrm{max}(1, r).$  
 \end{proof}
 \begin{corollary}$\label{a''}$
  Let  $ p_c(z) = z^n + c\overline{z}^k - 1 $ where  $c\in \mathbb{R^+},~~1 \leq k \leq n - 1,~~ n \geq 3,$ and $\mathrm{gcd}(n, k) = 1$ and $r \neq 1$ be  positive real root of equation $x^{n+1} -(1+M)x^n+M=0$ where $M:= \mathrm{max}\left( 1,|c|\right).$ Then all the zeros of $p_c(z)$ lie in the closed disc $D(0,R)$  where $R = \mathrm{max}(1, r).$
 \end{corollary}
\begin{proof}
 This corollary follows directly from Theorem $(\ref{a'})$ since it bounds all the zeros of any complex-valued harmonic polynomials. Now $$|p_c(z)| = |z^n+c\overline{z}^k -1| \geq |z|^n-M|z|^k+M \geq |z|^n-(1+M)|z|^k+M.$$ This shows that $|p_c(z)| > 0$ for $|z|>R.$ Hence the result.
 \end{proof}
 \subsection{Sharper bounds for the zeros of harmonic trinomials}$\label{p''}$
 In the previous section, we have found the upper bound on the moduli of all the zeros of complex-valued harmonic polynomials stated in Theorem $\ref{a'}$.  In the following theorem we find upper and lower bounds for the moduli of the zeros of harmonic trinoimial $p_c(z) = z^n + c\overline{z}^k - 1$ by considering  different cases.
  \begin{theorem}$\label{a}$
 Let  $ p_c(z) = z^n + c\overline{z}^k - 1 $ where  $1 \leq k \leq n - 1,~~ n \geq 3,$ and $\mathrm{gcd}(n, k) = 1.$ Then the zeros of $p_c(z)$ lies in the annular region 
 \begin{itemize}
 \item[(a)]  $(1-c)^{\frac{1}{n-k}} < |z| < (1+c)^{\frac{1}{n-k}}$ if $0<c<1.$
 \item[(b)] $(c-1)^{\frac{1}{n-k}} \leq |z| \leq (1+c)^{\frac{1}{n-k}}$ if $z$ is a zero of $p_c(z)$ with $|z| \geq 1$ and  $c\geq1.$
\item[(c)] $ (|c|-1)^{\frac{1}{n-k}} \leq |z| \leq (1+|c|)^{\frac{1}{n-k}}$   if $z$ is a zero of $p_c(z)$ with $|z| \geq 1$ and $ c \in \mathbb{C}.$
\end{itemize}
 \end{theorem}
We will prove this theorem after the following lemmas.
 The following two lemmas, lemma $(\ref{b})$  and lemma $(\ref{c}),$ are due to Kennedy  \cite{kennedy1940bounds}.
 \begin{lemma} $\label{b}$
 If $x_1$ and $x_2$ are positive real roots of the equations 
 $$ x^n+|a|x^k-|b| = 0~~ and ~~ x^n-|a|x^k-|b| = 0 $$ respectively, then the roots of $z^n+az^k+b=0,$ lies in the ring $x_1 \leq |z| \leq x_2$ where $ab \neq 0 ,~~ 1 \leq k \leq n - 1,~~ n \geq 3, $ and $\mathrm{gcd}(n, k) = 1.$
 \end{lemma}
 Note that in accordance with work of Anderson \textit{et al} \cite{anderson1998descartes} and Wang \cite{wang2004simple} both trinomial equations in lemma $(\ref{b})$ have only one positive real root.
 \begin{lemma} $\label{c}$
 If $y_1$ and $y_2$ are positive real roots of the equations $$ x^{n-k} - |b|^{\frac{(n-k)}{n}} + |a|=0 ~~ and ~~  x^{n-k} - |b|^{\frac{(n-k)}{n}} - |a|=0 $$ respectively, then the roots of $z^n+az^k+b=0,$ lies in the ring $y_1 < |z| < y_2$  where $ab \neq 0 ,~~ 1 \leq k \leq n - 1,~~ n \geq 3, $ and $\mathrm{gcd}(n, k) = 1.$
 \end{lemma}
Note that the roots of trinomial equation $z^n+az^k+b=0$ and the roots of its harmonic equivalent are bounded in the same ring with respect to their modulus.  In particular, since $ p_c(z) = z^n + c\overline{z}^k - 1 =0$ is a harmonic equivalent to $q(z)=z^n+az^k+b=0,$ the zero inclusion regions of $p$ and $q$ are the same.\\
Combining lemma $(\ref{e})$ and lemma $(\ref{f})$ plays a crucial role to prove Theorem $\ref{a}.$
 \begin{lemma}$\label{e}$
  Let $\alpha _1$ and $\alpha _2$ be positive real roots of  the equations $$ x^n+ cx^k-1=0 ~~ and~~ x^n-cx^k-1=0 $$ respectively,  where  $1 \leq k \leq n - 1,~~ n \geq 3, ~~0<c<1,$ and $\mathrm{gcd}(n, k) = 1.$ Then $p_c(z) = z^n + c\overline{z}^k - 1 $ has n distinct zeros and all the zeros lies in the ring $\alpha _1 \leq |z| \leq \alpha _2.$
 \end{lemma}
 \begin{proof}
 It was shown in \cite{brilleslyper2020zeros} that $p_c(z)$ has n distinct zeros. 
From $p_c(z) = z^n + c\overline{z}^k - 1 , $ we have by triangle inequality
$|z^n + c\overline{z}^k| = 1 \leq |z|^n + c|\overline{z}|^k = |z|^n+c|z|^k .$ Now then consider lemma $(\ref{b})$. By taking the values  $a = c$ and $b=-1$, the forward substitution helps us to finish the desired proof.
 \end{proof}
 Lemma $(\ref{c})$ also helps us to conclude the following result. Here by following the same procedures, we show that  $\beta_1 <\alpha_1$ and $\beta_2 > \alpha_2$.
  \begin{lemma}$\label{f}$
  The root of harmonic trinomial $p_c(z) = z^n + c\overline{z}^k - 1 $ lies in the ring $\beta _1 < |z| < \beta _2$ where $\beta _1$ and $\beta_2$ are the positive real roots of the equations $ x^{n-k}+c-1 =0$ and $x^{n-k}-c-1=0$  respectively. 
  \end{lemma}
  \begin{proof}
 It will suffices to show that $\beta_1 <\alpha_1$ and $\beta_2 > \alpha_2$, where $\alpha _1$ and $\alpha _2$ are positive real roots of  the equations $ x^n+ cx^k-1=0$ and $x^n-cx^k-1=0 $ respectively. Define $\mathcal{F}(x)= x^n+cx^k-1$ and $\mathcal{G}(x)=x^n-cx^k-1.$ Then 
$$ \mathcal{F}(\beta_1)= \beta_1^n +c\beta_1 ^k -1= \beta_1^k(\beta_1^{n-k}+c)-1 = \beta_1^k -1 < 1-1 = 0$$ since $\beta _1$ is a root of and $\beta_1 <1.$ This directly implies that $\mathcal{F}(\beta_1) < 0 = \mathcal{F}(\alpha_1).$ By applying the inverse function $(\mathcal{F})^{-1}$ to both sides of the last inequality, we find that  $\beta_1 < \alpha_1.$ Similarly, $ \mathcal{G}(\alpha_2) < \mathcal{G}(\beta_2)$ and thus $\alpha_2 < \beta_2.$ Therefore we conclude that the root of harmonic trinomial $p_c(z) = z^n + c\overline{z}^k - 1 $ lies in the ring $\beta _1 < |z| < \beta _2.$
  \end{proof}
 Now we are in position to prove our theorem, Theorem $\ref{a}.$
 \begin{proof}
 By looking at lemma $(\ref{f})$, the forward calculation gives us  $\beta_1 = (1-c)^{\frac{1}{n-k}}$ and $ \beta_2 = (1+c)^{\frac{1}{n-k}}.$ Moreover, we have $ (1-c)^{\frac{1}{n-k}} < |z| < (1+c)^{\frac{1}{n-k}}$ which ends the proof of part \textit{(a).} On one hand,  $p_c(z) = 0 \Rightarrow 1= |z^k||z^{n-k}+c| \geq |z|^k \left[ c-|z|^{n-k} \right].$ Since $|z| \geq 1,$ we have $1 \geq c- |z|^{n-k}.$ This show us that $\left(c-1\right)^{\frac{1}{n-k}} \leq |z|.$ On the other hand, $p_c(z) = 0 \Rightarrow 1= |z^k||z^{n-k}+c| \geq |z|^k \left[ |z|^{n-k}-c \right] \geq |z|^{n-k}-c.$ This gives us $\left(1+c\right)^{\frac{1}{n-k}} \geq |z|.$ Therefore, part \textit{(b)} follows. In similar fashion, part \textit{(c)} can also be proved.
 \end{proof}
 \begin{example}
Consider $p_c(z)=z^5+c\overline{z}^3-1.$ Here, if $n=5$ and $k=3,$ then $n-k=2.$ When $c=\frac{1}{2},$ corresponding to the case $0<c<1,$ $ \sqrt{\frac{1}{2}}$ and $ \sqrt{\frac{3}{2}}$ are the positive real roots of the equations $x^2 +\frac{1}{2}-1= x^2-\frac{1}{2}=0$ and $x^2-\frac{1}{2}-1= x^2-\frac{3}{2}=0$ respectively. Therefore all the zeros of  $p_c(z)= z^5+\frac{1}{2}\overline{z}^3-1$ are contained in $\sqrt{\frac{1}{2}} \leq |z| \leq \sqrt{\frac{3}{2}}.$ When $c=2,$ corresponding to the case $c \geq 1, 1$ and $\sqrt{3}$ are the positive real roots of the equations $x^2+2-1=x^2-1=0$ and $x^2-2-1=x^2-3=0$ respectively. Therefore all the zeros of  $p_c(z)= z^5+2\overline{z}^3-1$ are contained in $1 \leq |z| \leq \sqrt{3}.$ 
\end{example}
 \section{Conclusion}$\label{c'}$
 In this paper we have found an upper bound  for the moduli of all the zeros of complex-valued harmonic polynomials.
We derived here a closed disk in which all zeros are included and we bound the moduli of the zeros of the complex-valued harmonic  polynomials in general. More specifically,  it has been shown that the number of zeros of $p_c(z) = z^n + c\overline{z}^k - 1$ changes as $c$ varies by Brilleslyper \textit{et al} \cite{brilleslyper2020zeros}. In this paper  we have derived a locations for the zeros of $p_c(z)$ and the result  shows that the location of the zeros also changes as $c$ varies.
\section*{Acknowledgments}
The authors gratefully acknowledge Addis Ababa University for providing access to the completion of this work. The authors are supported in part by a grant from Simon's Foundation.
\section*{Declaration of Interest of Statement}
The authors declare that there are no conflicts of interest regarding the publication of this paper.

\end{document}